\documentclass{birkjour}
\usepackage[noadjust]{cite}
\usepackage{amsmath,amssymb,amsthm}
\usepackage{amsfonts}
\usepackage{mathrsfs}
\usepackage{hyperref}
\usepackage{bm}
\usepackage{graphicx}
\usepackage{tikz}
\usetikzlibrary{arrows,shapes,chains}
 \newtheorem{thm}{Theorem}[section]

 \newtheorem{prop}[thm]{Proposition}
 \theoremstyle{definition}
 \newtheorem{defn}[thm]{Definition}
 \theoremstyle{remark}
 \newtheorem{rem}[thm]{Remark}

 \newtheorem*{ex}{Example}
 \numberwithin{equation}{section}

\newcommand{\ed}{\end{document}}
\begin{document}
\thispagestyle{empty}
\tikzstyle{startstop}=[rectangle, rounded corners, minimum width = 2cm, minimum height=1cm,text centered, draw = black]
\tikzstyle{io}=[trapezium, trapezium left angle=70, trapezium right angle=110, minimum width=2cm, minimum height=1cm, text centered, draw=black]
\tikzstyle{process}=[rectangle, minimum width=3cm, minimum height=1cm, text centered, draw=black]
\tikzstyle{decision}=[diamond, aspect = 3, text centered, draw=black]
\tikzstyle{arrow} = [->,>=stealth]
%
%
%
%
%
%
%
%
%

\title[3-D analytic signal associated with LCT]
 {3-D generalized analytic signal associated with linear canonical transform in Clifford biquaternion domain}

\author{Zhen Feng Cai}

\address{%
School of Science, Hubei University of Technology\\
No.28, Nanli Road, Hong-shan District\\
Wuhan. Hubei\\
China}

\email{cai-zhenfeng@qq.com}

\author{Kit Ian Kou}
\address{Faulty of Science and Technology, University of Macau\\
Taipa, Macao\\
China}
\email{kikou@um.edu.mo}



\begin{abstract}
The analytic signal is a useful mathematical tool. It separates qualitative and quantitative information of a signal in form of the local phase and local amplitude. The Clifford Fourier transform (CFT) plays a vital role in the representation of multidimensional signals. By generalizing the CFT to the Clifford linear canonical transform (CLCT), we present a new type of Clifford biquaternionic analytic signal. Due to the advantages of more freedom, the envelop detection problems of 3D images, with the help of this new analytic signal, can get a better visual appearance. Synthesis examples are presented to demonstrate these advantages.
\end{abstract}

\maketitle
\section{Introduction}
The analytic signal, which was introduced by Gabor \cite{Gabor46} and Ville \cite{Ville48}, is a powerful tool in various applications such as communication \cite{Boashash92}, radar-based object detection \cite{Levanon04}, processing of oceanic data \cite{Lilly06}, etc. It is a complex signal which derives from adding its Hilbert transform to the original real-valued signal. This complex signal also can be regarded as being constructed by suppressing all negative frequency components of the original real signal. Since one can separate the qualitative and quantitative information from the local phase and local amplitude of analytic signals, the instantaneous amplitude and phase are the hot topic of analytic signal analysis, and a critical analysis about it has been carried out by Picibono \cite{Picibono97} in 1997. During the last several decades, many new complex analytic signal theories \cite{sangwine07,Hahn11,Unser09,Said08,Felsberg01,Yang17} were well established with the development of Clifford algebras and the associated Fourier transform theory.

The linear canonical transform (LCT) - a powerful tool for optics and signal processing - was first introduced in the 1970s by Collins \cite{Collins70} and Moshinsky \cite{Moshinsky71}. It is a linear integral transform with three free parameters. Many famous transforms such as Fourier transform (FT), fractional transform (FRFT), and the Fresnel transform (FST) are all special cases of the LCT \cite{Wolf79, Ozaktas00, Pei01}. Due to the property that owning more degrees of freedom and needing similar computation cost to the FT and FRFT, the LCT has many applications such as in signal synthesis, radar system analysis, filter design and pattern recognition, etc \cite{Ozaktas00, Pei03}. Recently, with the development of the LCT, the analytic signal has been extended into the LCT domain initially by Fu and Li \cite{Fu08} and to 2D LCT domain by Xu et al \cite{Xu09}. Furthermore, Kou \cite{Kou16} generalized the analytic signal to the quaternion domain with the help of the two-sided quaternion linear canonical transform (QLCT) and got a satisfactory result by using this method to process the envelop detection problems.

In 3-D image processing domain, such as 3-D ultrasound image registration, analytic signal is a powerful tool. Zhang \cite{Zhang06} first used it to process 3-D ultrasound image with the help of its phase information. Following this, phase information of analytic signal is widely used in image processing \cite{Harput11,Maltaverne10,Rajpoot09,Belaid11}. Since local amplitude is another important part of analytic signal, Wang \cite{Wang12} introduced 3-D Clifford biquaternionic analytic signal, which is associated with Clifford Fourier transform. With the help of partial modules, it overcomes the shortcoming of losing information of the classical 1-D analytic envelop detection tools. Since the analytic signal in LCT domain supplies better results in envelop detection than that in FT domain, which was shown by Kou in \cite{Kou16}. We generalize Wang's \cite{Wang12} Clifford biquaternionic analytic signal to LCT domain, and with the help of the local amplitude, the envelopes of 3-D images are successfully detected.  Synthesis examples show that our approach presents better results than Wang's method. Furthermore, by comparing with the amplitude method of monogenic signal which is another kind of generalization of analytic signal, the powerful ability of our approach is verified.

The paper is organized as follows. First, we recall the basic knowledge about the Clifford  biquaternion and $n$ dimensional analytic signal in section 2. Section 3 is dedicated to give the definition and basic properties of CLCT, which are the keys of using the generalized analytic signal to detect the envelopes of 3-D images. In section 4, a novel approach of envelop detection based on CLCT is supplied and synthetic examples are introduced to show the advantages of this method. Finally we conclude this article in Section 5.

\section{Preliminary}
\subsection{Clifford Biquaternion}
Quaternion, which was discovered by Hamilton in 1843 and is denoted by $\mathbb{H}$, is a generalization of complex number. Each quaternion number $\mathbf{q}$ (denoted by bold letter in this paper) has a form $\mathbf{q}=q_0+q_1\textbf{i}+q_2\textbf{j}+q_3\textbf{k}$, where $q_0,q_1,q_2,q_3$ are real numbers and $\textbf{i}, \textbf{j}$, and $\textbf{k}$ are imaginary units, which satisfy $\textbf{i}^2=\textbf{j}^2=\textbf{k}^2=-1$ and $\textbf{ij}=-\textbf{ji}=\textbf{k}$. For every quaternion number $\mathbf{q}=q_0+q_1\textbf{i}+q_2\textbf{j}+q_3\textbf{k}$, the scalar part and vector part are $Sc(q):=q_0$ and $\underline{\mathbf{q}}:=q_1\textbf{i}+q_2\textbf{j}+q_3\textbf{k}$ respectively. We also use symbols $\overline{\mathbf{q}}:=q_0-\underline{\mathbf{q}}$ and $|\mathbf{q}|:=\sqrt{\mathbf{q}\overline{\mathbf{q}}}=\sqrt{q_0^2+q_1^2+q_2^2+q_3^2}$ to denote the conjugate and norm of $\mathbf{q}$.
Furthermore, for $p=1$ and $2$, the quaternion modules $L^p(\mathbb{R}^n,\mathbb{H})$ is defined by
\begin{equation*}
 L^p(\mathbb{R}^n,\mathbb{H}):=\{f|f:\mathbb{R}^n\rightarrow\mathbb{H},\|f\|_{L^p(\mathbb{R}^n,\mathbb{H})}:=\big(\int_{\mathbb{R}^n}|f(\mathbf{x})|^pd\mathbf{x}\big)^{\frac{1}{p}}<\infty\}.
\end{equation*}

Clifford algebra, which is a further generalization of quaternion, was introduced by William Kingdon Clifford in 1878 and is defined as follows.
\begin{defn}\label{def-clif}
Suppose that $\{\mathbf{e}_1,\mathbf{e}_2,\cdots,\mathbf{e}_n\}$ is an orthonormal basis of Euclidean space $\mathbb{R}^n$, and satisfies the relations $\mathbf{e}_i^2=-1$ for $i=1,2,\cdots,n,$ and $\mathbf{e}_i\mathbf{e}_j+\mathbf{e}_j\mathbf{e}_i=0$ for $1\leq i\neq j\leq n.$ Then the Clifford algebra $Cl(0,n)$ is an algebra constructed over these elements, i.e.,
\begin{equation}\label{CliffordDF}
Cl(0,n):=\bigg\{a=\sum\limits_S a_S\mathbf{e}_S: a_S\in \mathbb{R}, \mathbf{e}_S=\mathbf{e}_{j_1}\mathbf{e}_{j_2}\cdots \mathbf{e}_{j_k}\bigg\},
\end{equation}
where $S:=\{j_1,j_2,\cdots,j_k\}\subseteq\{1,2,\cdots,n\}$ with $1\leq j_1< j_2<\cdots< j_k\leq n$; or $S=\emptyset,$ and $\mathbf{e}_{\emptyset}:=1$.
\end{defn}

It can be easily found from the definition that $Cl(0,n)$ is a $2n$ dimensional real linear vector space. Let $|S|$ be the number of elements in the set $S$. For each $k\in\{1,2,\cdots,n\},$ let
\begin{equation*}
Cl^{(k)}(0,n):=\bigg\{a\in Cl(0,n): a=\sum\limits_{|S|=k} a_S\mathbf{e}_S \bigg\},
\end{equation*}
to be the $k$-vectors subset of $Cl(0,n)$.
Hence we have
\begin{equation*}
Cl(0,n)=\bigoplus_{k=0}^{n} Cl^{(k)}(0,n).
\end{equation*}
For any Clifford number $a=\sum\limits_{S}a_{S}\mathbf{e}_{S}$ in $Cl(0,n)$, it has a projection $[a]_k$ on $Cl^{(k)}(0,n)$, and can be represented by
\begin{equation*}
a=\sum\limits_{k=0}^{n} [a]_k.
\end{equation*}
For $k=0$, $[a]_0$ is named the scalar part of $a$. Furthermore $[a]_1$, $[a]_2$ and $[a]_n$, are the vector part, bivector part and pseudoscalar part of $a$ respectively.

Similar to quaternion numbers, we use symbols $\overline{a}:=\sum\limits_{S}a_{S}\overline{\mathbf{e}}_{S}$ and $|a|:=\left(\sum\limits_S|a_S|^2\right)^{1/2}$ to denote the conjugate and norm of $a$, where $\overline{\mathbf{e}}_{S}:=(-1)^{\frac{|S|(|S|+1)}{2}}\mathbf{e}_{S}$. Furthermore the Clifford  modules $L^p(\mathbb{R}^n,Cl(0,n))$ is
\begin{equation*}
 L^p(\mathbb{R}^n,Cl(0,n)):=\{f|f:\mathbb{R}^n\rightarrow Cl(0,n),\|f\|_{L^p(\mathbb{R}^n,Cl(0,n))}:=\big(\int_{\mathbb{R}^n}|f(\mathbf{x})|^pd\mathbf{x}\big)^{\frac{1}{p}}<\infty\}.
\end{equation*}
For $p=2$, and let $f,g \in L^2(\mathbb{R}^n,Cl(0,n))$, an inner product can be equipped as follows
\begin{equation}\label{innerp}
(f,g)_{L^2(\mathbb{R}^n,Cl(0,n))}:=\int_{\mathbb{R}^n}f(\mathbf{x})\overline{g(\mathbf{x})}d^n\mathbf{x}.
\end{equation}
Since Clifford algebra is generalized from quaternion algebra, they have a close relationship as follows,
\begin{thm}\cite{Girard11}\label{The-clif}
If $p+q=2m$ ($m$ is integer),the Clifford algebra $Cl(p,q)$ is the tensor product of $m$ quaternion algebras. If $p+q=2m-1$, the Clifford algebra $Cl(p,q)$ is the tensor product of $m-1$ quaternion algebra and the algebra $(1,\epsilon)$, where $\epsilon$ is the product of the $2m-1$ generators $(\epsilon=\mathbf{e}_1\mathbf{e}_2\cdots\mathbf{e}_{2m-1})$.
\end{thm}

According to the above theorem, one can easily find some special examples of Clifford algebras such as: complex $\mathbb{C}$ with $p=0, q=1, m=1$ and $\mathbf{e}_1=\mathbf{i}$; quaternions $\mathbb{H}$ with $p=0, q=2, m=1$ and  $\mathbf{e}_1=\mathbf{i}, \mathbf{e}_2=\mathbf{j}$.

Here we should pay an attention to a special case of Clifford algebra, $Cl(0,3)$, which will be used in this paper to detect the envelop of 3-D images. According to Theorem \ref{The-clif}, $Cl(0,3)$ corresponds $p=0, q=3,$ and is a tensor product of quaternion algebra $\mathbb{H}$ and the algebra $(1,\epsilon)$, where $\mathbf{e}_1=\epsilon\mathbf{i}, \mathbf{e}_2=\epsilon\mathbf{j},\mathbf{e}_3=\epsilon\mathbf{k}$ are the generators of $Cl(0,3)$, $\epsilon=\mathbf{e}_1\mathbf{e}_2\mathbf{e}_3$ commuting with $\mathbf{i,j,k}$ and $\epsilon^2=1$.
Thus each number in this algebra
$$A=a_0+a_1\mathbf{e}_1+a_2\mathbf{e}_2+a_3\mathbf{e}_3+a_4\mathbf{e}_1\mathbf{e}_2+a_5\mathbf{e}_3\mathbf{e}_1+a_6\mathbf{e}_2\mathbf{e}_3+a_7\mathbf{e}_1\mathbf{e}_2\mathbf{e}_3,$$
can be expressed by $A=\mathbf{p}+\epsilon\mathbf{q}$, where $\mathbf{p},\mathbf{q}$ are two quaternions and
\begin{equation*}
\begin{aligned}
\mathbf{p}&=(a_0+a_6\mathbf{i}+a_5\mathbf{j}+a_4\mathbf{k}),
\\
\mathbf{q}&=(a_7+a_1\mathbf{i}+a_2\mathbf{j}+a_3\mathbf{k}).
\end{aligned}
\end{equation*}
Here $a_0$ is the scalar part of $A$, $a_1,a_2,a_3$ correspond to the vector part, and $a_4,a_5,a_6$ are the bivector parts, while $a_7$ is the pseudoscalar part of $A$.
\begin{rem}\label{clibiquaternion}
Since $Cl(0,3)$ is a special case of geometric algebra (Clifford algebra) and isomorphic to $\mathbb{H}\bigoplus\mathbb{H}$, it is also named Clifford biquaternion\cite{Wang12,Girard11} (clifbquat for short) by some authors.
\end{rem}
 Since Clifford biquaternion algebra $Cl(0,3)$ is a special case of geometric algebra (Clifford algebra), it follows the reflecting properties of geometric algebra \cite{Vince08}, which states that under an orthogonal symmetry with respect to a plane which is perpendicular to a unit vector $a$, the reflecting of a Clifford biquaterion $A=\sum\limits_{|S|=0}^3 a_S\mathbf{e}_S$ is $A'=\sum\limits_{|S|=0}^3 (-1)^{|S|+1}a_S a\mathbf{e}_Sa$.
In particular, if $a=\mathbf{e}_1$, we can get 
\begin{equation}
A'=K_1(A)=(a_0+a_6\mathbf{i}-a_5\mathbf{j}-a_4\mathbf{k})+\epsilon(-a_7-a_1\mathbf{i}+a_2\mathbf{j}+a_3\mathbf{k});
\end{equation}
if $a=\mathbf{e}_2$, it is
\begin{equation}
A'=K_2(A)=(a_0-a_6\mathbf{i}+a_5\mathbf{j}-a_4\mathbf{k})+\epsilon(-a_7+a_1\mathbf{i}-a_2\mathbf{j}+a_3\mathbf{k});
\end{equation}
if $a=\mathbf{e}_3$,
\begin{equation}
A'=K_3(A)=(a_0-a_6\mathbf{i}-a_5\mathbf{j}+a_4\mathbf{k})+\epsilon(-a_7+a_1\mathbf{i}+a_2\mathbf{j}-a_3\mathbf{k}).
\end{equation}

\subsection{Analytic Signal in N Dimension}
Analytic signal is constructed by suppressing the negative frequency components of the original signal. So it is closely related with the Fourier transform of the original signal. In \cite{Girard11}, Girard generalized the analytic signal to $n$ dimension by introducing a new Clifford-Fourier transform as follows
\begin{defn}Given a Clifford valued function $f(\mathbf{x})\in L^1(\mathbb{R}^n,Cl(0,n))$ with $\mathbf{x}=(x_1,x_2,\cdots,x_n)$, its Clifford Fourier transform $F(u)$ is defined as
\begin{equation}
F(\mathbf{u})=\int_{\mathbb{R}^n}f(\mathbf{x})\prod\limits_{k=1}^{n}e^{-\mathbf{e}_k2\pi u_kx_k}d^n\mathbf{x}.
\end{equation}
\end{defn}

Furthermore, Girard supplied the inverse Clifford Fourier transform as follows
\begin{equation}
f(\mathbf{x})=\int_{\mathbb{R}^n}F(\mathbf{u})\prod\limits_{k=0}^{n-1}e^{\mathbf{e}_{n-k}2\pi u_{n-k}x_{n-k}}d^{n}\mathbf{u}.
\end{equation}

With the help of this CFT, Girard introduced the $n$ dimensional analytic signal $f_A(\mathbf{x})$ which corresponds to $f(\mathbf{x})$  by the following steps
\begin{equation}
F_A(\mathbf{u})=\prod\limits_{k=1}^{n}[1+sign(u_k)]F(\mathbf{u}),
\end{equation}
\begin{equation}
f_A(\mathbf{x})=\int_{\mathbb{R}^n}F_A(\mathbf{u})\prod\limits_{k=0}^{n-1}e^{\mathbf{e}_{n-k}2\pi u_{n-k}x_{n-k}}d^n\mathbf{u},
\end{equation}
where $sign(u_k)$ is the classical signum function and $F(\mathbf{u})$ is the CFT of original function $f(\mathbf{x})$.

It can be easily verified that when $n=1$, CFT degenerates to the classical FT and the corresponding analytic signal is $f_A(t)$. When $n=2$, the CFT turns to the right-sided quaternion Fourier transform and the analytic signal is quaternionic valued function $f_A(x,y)$.

Since LCT is a generalization of FT, one can generalize the analytic signal by replacing the FT by LCT, which was introduced by Fu in \cite{Fu08}. One can connect these two approaches to generalize the analytic signal to $n$ dimensional LCT domain.

Monogenic signal which was introduced by  Felsberg \cite{Felsberg01} is another generalization of analytic signal. Since we will compare our method with that of monogenic signal, the basic knowledge about monogenic signal will be reviewed in the following.

\begin{defn}\cite{Yang17}(Monogenic signal) For $f\in L^2(\mathbb{R}^n,Cl(0,n))$, the monogenic signal $f_M\in L^2(\mathbb{R}^n,Cl(0,n))$ is defined by
\begin{equation}\label{mono-defn}
f_M(\underline{x}):=f(\underline{x})+H[f](\underline{x}),
\end{equation}
where $H[f]$ is the isotropic Hilbert transform of $f$ defined by
\begin{equation}
\begin{aligned}
\label{hilber-defn}
H[f](\underline{x}):&=p.v.\frac{1}{\omega_n}\int_{\mathbf{R}^n}\frac{\overline{\underline{x}-\underline{t}}}{|\underline{x}-\underline{t}|^{n+1}}f(\underline{t})d\underline{t}
\\
&=\lim\limits_{\epsilon\rightarrow0^+}\frac{1}{\omega_n}\int_{{|\underline{x}-\underline{t}}|>\epsilon}\frac{\overline{\underline{x}-\underline{t}}}{|\underline{x}-\underline{t}|^{n+1}}f(\underline{t})d\underline{t}
\\
&=-\sum\limits_{j=1}^{n}R_j(f)(\underline{x})\mathbf{e}_j.
\end{aligned}
\end{equation}
Furthermore,
\begin{equation*}
R_j(f)(\underline{x}):=\lim\limits_{\epsilon\rightarrow0^+}\frac{1}{\omega_n}\int_{{|\underline{x}-\underline{t}}|>\epsilon}\frac{x_j-t_j}{|\underline{x}-\underline{t}|^{n+1}}f(\underline{t})d\underline{t}
\end{equation*}
and $\omega_n=\frac{2\pi^{\frac{n+1}{2}}}{\Gamma(\frac{n+1}{2})}$.
\end{defn}

\section{Clifford Linear Canonical Transform}
Since LCT is a generalization of many famous integral transform and has more degrees of freedom, it is natural to generalize the LCT to Clifford domain. In the last decade, many mathematicians tried different approaches to develop this kind of generalization. In \cite{Kou13}, Kou introduced the CLCT of a function $f\in L^1(\mathbb{R}^n,Cl(0,n))$ and investigated the maxima of energy preservation problem. Yang \cite{Yang14} paid attention to the CLCT with the kernel constituted by the complex unit $I$. Here, we will introduce a new type of generalization as follows
\begin{defn} (Right-sided CLCT)
Let $\Lambda$: $A_k=\left(
\begin{array}{cc}
a_k & b_k \\
c_k & d_k \\
\end{array}
\right)\in \mathbb{R}^{2\times2}$
($k=1,2,\cdots,n$) be a set of parameter matrices  with $det(A_k)=1$ and $b_k\neq0$, the right-sided CLCT of a signal $f\in L^1(\mathbb{R}^n,CL(0,n))$ is a Clifford valued function $\mathscr{L}^r_{\Lambda}(f): \mathbb{R}^n\rightarrow CL(0,n)$ defined as follows
\begin{equation}\label{def-rCLCT}
\mathscr{L}^r_{\Lambda}(f)(\mathbf{u}):=\int_{\mathbb{R}^n}f(\mathbf{x})\prod\limits_{k=1}^{n} K^{\mathbf{e}_k}_{A_k}(x_k,u_k)d^n\mathbf{x},
\end{equation}
where $K^{\mathbf{e}_k}_{A_k}(x_k,u_k)$ is the kernel of the CLCT defined as
\begin{equation}
K^{\mathbf{e}_k}_{A_k}(x_k,u_k):=\frac{1}{\sqrt{\mathbf{e}_k2\pi b_k}}e^{\mathbf{e}_k(\frac{a_k}{2b_k}x_k^2-\frac{1}{b_k}x_ku_k+\frac{d_k}{2b_k}u_k^2)}.
\end{equation}
\end{defn}

Due to the non-commutativity of Clifford algebra, there is a different type of CLCT, left-sided CLCT, which can be defined as follows
\begin{defn} (Left-sided CLCT)
Let $\Lambda$: $A_k=\left(
\begin{array}{cc}
a_k & b_k \\
c_k & d_k \\
\end{array}
\right)\in \mathbb{R}^{2\times2}$
($k=1,2,\cdots,n$) be a set of parameter matrices  with $det(A_k)=1$ and $b_k\neq0$, the left-sided CLCT of a signal $f\in L^1(\mathbb{R}^n,Cl(0,n))$ is a Clifford valued function $\mathscr{L}^l_{\Lambda}(f): \mathbb{R}^n\rightarrow Cl(0,n)$ defined as follows
\begin{equation}
\mathscr{L}^l_{\Lambda}(f)(\mathbf{u}):=\int_{\mathbb{R}^n}\prod\limits_{k=1}^{n} K^{\mathbf{e}_k}_{A_k}(x_k,u_k)f(\mathbf{x})d^n\mathbf{x}.
\end{equation}
\end{defn}

Here we focus on the right-sided CLCT in the following content. Furthermore, we can simulate the approach Hitzer used in \cite{Hitzer14} to introduce a two-sided Clifford LCT with two square roots of $-1$ in $Cl(0,n)$, which is out of our scope.

\begin{rem}It can be easily verified that when the dimension $n=1$, the CLCT degenerates to the classical LCT. When $n=2$, the CLCT turns to the right-sided quaternion linear canonical transform.
\end{rem}

Now let us investigate an example to see how the CLCT works:
\begin{ex}
Considering the right-sided CLCT of $\prod\limits_{k=0}^{n-1} K^{-\mathbf{e}_{n-k}}_{A_{n-k}}(x_{n-k},v_{n-k})$.
\end{ex}
Since $f(t)$ is product of $n$ LCT kernel functions, we can integrate the functions separately when taking the CLCT.
That is
\begin{equation}\label{ExamK}
\begin{split}
&\mathscr{L}^r_{\Lambda}(f)(\mathbf{u})=\int_{\mathbb{R}^n}\prod\limits_{k=0}^{n-1} K^{-\mathbf{e}_{n-k}}_{A_{n-k}}(x_{n-k},v_{n-k})\prod\limits_{k=1}^{n} K^{\mathbf{e}_k}_{A_k}(x_k,u_k)d^n\mathbf{x}
\\
&=\int_{\mathbb{R}^{n-1}}\prod\limits_{k=0}^{n-2} K^{-\mathbf{e}_{n-k}}_{A_{n-k}}(x_{n-k},v_{n-k})\int_{\mathbb{R}}K^{-\mathbf{e}_{1}}_{A_{1}}(x_{1},v_{1})K^{\mathbf{e}_{1}}_{A_{1}}(x_{1},u_{1})dx_1
\\
&\cdot\prod\limits_{k=2}^{n} K^{\mathbf{e}_k}_{A_k}(x_k,u_k)d^{n-1}\mathbf{x},
\end{split}
\end{equation}
where
\begin{equation}\label{Delta}
\begin{split}
\int_{\mathbb{R}}K^{-\mathbf{e}_{1}}_{A_{1}}(x_{1},v_{1})K^{\mathbf{e}_{1}}_{A_{1}}(x_{1},u_{1})&dx_1=\int_{\mathbb{R}}\frac{1}{\sqrt{-\mathbf{e}_12\pi b_1}}e^{-\mathbf{e}_1(\frac{a_1}{2b_1}x_1^2-\frac{1}{b_1}x_1v_1+\frac{d_1}{2b_1}v_1^2)}
\\
&\cdot\frac{1}{\sqrt{\mathbf{e}_12\pi b_1}}e^{\mathbf{e}_1(\frac{a_1}{2b_1}x_1^2-\frac{1}{b_1}x_1u_1+\frac{d_1}{2b_1}u_1^2)}dx_1
\\
&=\frac{1}{2\pi}\int_{\mathbb{R}}e^{-\mathbf{e}_1\frac{x_1}{b_1}(u_1-v_1)}d(\frac{x_1}{b_1})e^{-\mathbf{e}_1\frac{d_1}{2b_1}(u_1^2-v_1^2)}
\\
&=\delta(u_1-v_1).
\end{split}
\end{equation}
The last step of the above equation comes from the relationship of Fourier transform $\widehat{1}(\omega)=2\pi \delta(\omega)$.
Applying result (\ref{Delta}) to the right side of equation (\ref{ExamK}), one can derive
$$\mathscr{L}^r_{\Lambda}\big(\prod\limits_{k=0}^{n-1} K^{-\mathbf{e}_{n-k}}_{A_{n-k}}(x_{n-k},v_{n-k})\big)(\mathbf{u})=\prod\limits_{k=1}^{n}\delta(u_k-v_k).$$

In the following content, we turn to investigate the main properties of CLCT.
\begin{prop} (Left linearity) Let $\alpha \in Cl(0,n), \beta \in Cl(0,n)$, for $f_1(\mathbf{x}),f_2(\mathbf{x})\in L^1(\mathbb{R}^n, Cl(0,n))$, we have $\mathscr{L}^r_{\Lambda}(\alpha f_1+\beta f_2)(\mathbf{u})=\alpha\mathscr{L}^r_{\Lambda} (f_1)(\mathbf{u})+\beta \mathscr{L}^r_{\Lambda} (f_2)(\mathbf{u})$.
\end{prop}

This proposition can be easily verified by the linearity of integral.

\begin{prop}(Translation) Let $\mathbf{\bm{\alpha}}=(\alpha_1,0,\cdots,0,\alpha_n)$, for $f(\mathbf{x})\in L^1(\mathbb{R}^n,\mathbb{R})$, we have
$
\mathscr{L}^r_{\Lambda}(f(\mathbf{x}-\bm{\alpha}))(\mathbf{u})=e^{\mathbf{e}_1(c_1u_1\alpha_1)}e^{-\mathbf{e}_1\frac{a_1c_1}{2}\alpha_1^2}\mathscr{L}^r_{\Lambda}(f)(u_1-a_1\alpha_1,u_2,\cdots,u_n-a_n\alpha_n)e^{\mathbf{e}_n(c_nu_n\alpha_n)}e^{-\mathbf{e}_n\frac{a_nc_n}{2}\alpha_n^2}.
$
\end{prop}
\begin{proof}
According to definition of right-sided CLCT, one can get
\begin{equation}\label{trans1}
\mathscr{L}^r_{\Lambda}(f(\mathbf{x}-\bm{\alpha}))(\mathbf{u})=\int_{\mathbb{R}^n}f(x_1-\alpha_1,x_2,\cdots,x_{n-1},x_n-\alpha_n)\prod\limits_{k=1}^{n} K^{\mathbf{e}_k}_{A_k}(x_k,u_k)d^n\mathbf{x}.
\end{equation}
Let $t_1=x_1-\alpha_1$, $t_n=x_n-\alpha_n$, the above integral turns to a new integral about variables $t_1,x_2,\cdots,x_{n-1},t_n$, where the first LCT kernel becomes
\begin{equation*}
\begin{aligned}
K^{\mathbf{e}_1}_{A_1}(x_1,u_1)&=\frac{1}{\sqrt{\mathbf{e}_12\pi b_1}}e^{\mathbf{e}_1[\frac{a_1}{2b_1}(t_1+\alpha_1)^2-\frac{1}{b_1}(t_1+\alpha_1)u_1+\frac{d_1}{2b_1}u_1^2]}
\\
&=\frac{1}{\sqrt{\mathbf{e}_12\pi b_1}}e^{\mathbf{e}_1[\frac{a_1}{2b_1}t_1^2+\frac{a_1}{b_1}t_1\alpha_1+\frac{a_1}{2b_1}\alpha_1^2-\frac{t_1u_1}{b_1}-\frac{u_1\alpha_1}{b_1}+\frac{d_1}{2b_1}u_1^2]}
\\
&=e^{\mathbf{e}_1(c_1u_1\alpha_1)}e^{-\mathbf{e}_1\frac{a_1c_1}{2}\alpha_1^2}\frac{1}{\sqrt{\mathbf{e}_12\pi b_1}}e^{\mathbf{e}_1[\frac{a_1}{2b_1}t_1^2-\frac{t_1}{b_1}(u_1-a_1\alpha_1)+\frac{d_1}{2b_1}(u_1-a_1\alpha_1)^2]},
\end{aligned}
\end{equation*}
and the last LCT kernel turns to
\begin{equation*}
K^{\mathbf{e}_n}_{A_n}(x_n,u_n)=e^{\mathbf{e}_n(c_nu_n\alpha_n)}e^{-\mathbf{e}_n\frac{a_nc_n}{2}\alpha_n^2}\frac{1}{\sqrt{\mathbf{e}_n2\pi b_n}}e^{\mathbf{e}_n[\frac{a_n}{2b_n}t_n^2-\frac{t_n}{b_n}(u_n-a_n\alpha_n)+\frac{d_n}{2b_n}(u_n-a_n\alpha_n)^2]}.
\end{equation*}
Since $f(\mathbf{x})$ is real valued function in this case and it can interchange the position with $K^{\mathbf{e}_1}_{A_1}(x_1,u_1)$ in equation (\ref{trans1}) freely. Hence equation (\ref{trans1}) derives
\begin{equation*}
\begin{aligned}
\mathscr{L}^r_{\Lambda}(f(\mathbf{x}-\bm{\alpha}))(\mathbf{u})=&e^{\mathbf{e}_1(c_1u_1\alpha_1)}e^{-\mathbf{e}_1\frac{a_1c_1}{2}\alpha_1^2}\mathscr{L}^r_{\Lambda}(f)(u_1-a_1\alpha_1,u_2,\cdots,u_n-a_n\alpha_n)
\\
&e^{\mathbf{e}_n(c_nu_n\alpha_n)}e^{-\mathbf{e}_n\frac{a_nc_n}{2}\alpha_n^2}.
\end{aligned}
\end{equation*}
\end{proof}
%

\begin{prop}(Scaling)Let $\sigma_1,\sigma_2,\cdots,\sigma_n>0$, for $f(\mathbf{x})\in L^1(\mathbb{R}^n,Cl(0,n))$, we have
$$
\mathscr{L}^r_{\Lambda}(f(\sigma_1x_1,\sigma_2x_2,\cdots,\sigma_nx_n)(\mathbf{u})=\frac{1}{\sqrt{\prod\limits_{k=1}^{n}\sigma_k}}\mathscr{L}^r_{\Lambda'}(f)(\mathbf{u}),
$$
where $\Lambda'$: $B_k=\left(
\begin{array}{cc}
a_k/\sigma_k & b_k\sigma_k \\
c_k/\sigma_k & d_k\sigma_k \\
\end{array}
\right).$
\end{prop}
\begin{proof}According to equation (\ref{def-rCLCT}), one can find
\begin{equation}\label{Escaling}
\begin{split}
\mathscr{L}^r_{\Lambda}(f(\sigma_1x_1,\sigma_2x_2,\cdots,\sigma_nx_n))(\mathbf{u})=&\int_{\mathbb{R}^n}f(\sigma_1x_1,\sigma_2x_2,\cdots,\sigma_nx_n)
\\
&\cdot\prod\limits_{k=1}^{n} K^{\mathbf{e}_k}_{A_k}(x_k,u_k)d^n\mathbf{x}.
\end{split}
\end{equation}
Let $t_k=\sigma_kx_k\quad(k=1,\cdots,n)$, each kernel $K^{\mathbf{e}_k}_{A_k}(x_k,u_k)$ turns to
\begin{equation*}
\begin{aligned}
K^{\mathbf{e}_k}_{A_k}(x_k,u_k)&=\frac{1}{\sqrt{\mathbf{e}_k2\pi b_k}}e^{\mathbf{e}_k[\frac{a_k}{2b_k}(\frac{t_k}{\sigma_k})^2-\frac{1}{b_k}(\frac{t_k}{\sigma_k})u_k+\frac{d_k}{2b_k}u_k^2]}
\\
&=\frac{\sqrt{{\sigma_k}}}{\sqrt{\mathbf{e}_k2\pi b_k{\sigma_k}}}e^{\mathbf{e}_k[\frac{a_k/{\sigma_k}}{2b_k{\sigma_k}}t_k^2-\frac{1}{b_k{\sigma_k}}t_ku_k+\frac{d_k{\sigma_k}}{2b_k{\sigma_k}}u_k^2]}.
\end{aligned}
\end{equation*}
Hence the left side of equation (\ref{Escaling}) becomes
\begin{equation*}
\begin{split}
\mathscr{L}^r_{\Lambda}(f(t_1,t_2,\cdots,t_n))(\mathbf{u})&=\int_{\mathbb{R}^n}f(t_1,t_2,\cdots,t_n)
\\
&\cdot\prod\limits_{k=1}^{n} \frac{\sqrt{{\sigma_k}}}{\sqrt{\mathbf{e}_k2\pi b_k{\sigma_k}}}e^{\mathbf{e}_k[\frac{a_k/{\sigma_k}}{2b_k{\sigma_k}}t_k^2-\frac{1}{b_k{\sigma_k}}t_ku_k+\frac{d_k{\sigma_k}}{2b_k{\sigma_k}}u_k^2]}d{\frac{t_1}{\sigma_1}}\cdots d{\frac{t_n}{\sigma_n}}
\\
&=\frac{1}{\sqrt{\prod\limits_{k=1}^{n}\sigma_k}}\int_{\mathbb{R}^n}f(t_1,t_2,\cdots,t_n)\prod\limits_{k=1}^{n} \frac{1}{\sqrt{\mathbf{e}_k2\pi b_k{\sigma_k}}}
\\
&\cdot e^{\mathbf{e}_k[\frac{a_k/{\sigma_k}}{2b_k{\sigma_k}}t_k^2-\frac{1}{b_k{\sigma_k}}t_ku_k+\frac{d_k{\sigma_k}}{2b_k{\sigma_k}}u_k^2]}d^n\mathbf{t},
\end{split}
\end{equation*}
which is the desired result.
\end{proof}

\begin{prop}\label{Partial}(Partial derivative) For $f(\mathbf{x})$ and $\frac{\partial f(\mathbf{x})}{\partial x_1},\frac{\partial f(\mathbf{x})}{\partial x_n}\in L^1(\mathbb{R}^n,\mathbb{R})$, we have
\begin{equation}\label{Partial1}
\mathscr{L}^r_{\Lambda}(\frac{\partial f(\mathbf{x})}{\partial x_1})(\mathbf{u})=(a_1\frac{\partial}{\partial u_1}-\mathbf{e}_1c_1u_1)\mathscr{L}^r_{\Lambda}(f)(\mathbf{u}),
\end{equation}
and
\begin{equation}\label{Partialn}
\mathscr{L}^r_{\Lambda}(\frac{\partial f(\mathbf{x})}{\partial x_n})(\mathbf{u})=\mathscr{L}^r_{\Lambda}(f)(\mathbf{u})(a_n\frac{\partial}{\partial u_n}-\mathbf{e}_nc_nu_n).
\end{equation}
\end{prop}
\begin{proof}
Since  $\frac{\partial f(\mathbf{x})}{\partial x_1}\in L^1(\mathbb{R}^n,\mathbb{R})$, it has CLCT as following
\begin{equation*}
\begin{aligned}
\mathscr{L}^r_{\Lambda}(\frac{\partial f(\mathbf{x})}{\partial x_1})(\mathbf{u})&=\int_{\mathbb{R}^n}\frac{\partial f(\mathbf{x})}{\partial x_1}\prod\limits_{k=1}^{n} K^{\mathbf{e}_k}_{A_k}(x_k,u_k)d^n\mathbf{x}
\\
&=\int_{\mathbb{R}^{n-1}}\bigg[\int_\mathbb{R}\frac{\partial f(\mathbf{x})}{\partial x_1}K^{\mathbf{e}_1}_{A_1}(x_1,u_1)dx_1\bigg]\prod\limits_{k=2}^{n} K^{\mathbf{e}_k}_{A_k}(x_k,u_k)d^{n-1}\mathbf{x},
\end{aligned}
\end{equation*}
where
\begin{equation*}
\begin{aligned}
\int_\mathbb{R}\frac{\partial f(\mathbf{x})}{\partial x_1}&K^{\mathbf{e}_1}_{A_1}(x_1,u_1)dx_1=\big[f(\mathbf{x})K^{\mathbf{e}_1}_{A_1}(x_1,u_1)\big]\big|_{-\infty}^{+\infty}-\int_\mathbb{R} f(\mathbf{x})\frac{\partial}{\partial x_1}K^{\mathbf{e}_1}_{A_1}(x_1,u_1)dx_1
\\
&=-\int_\mathbb{R} f(\mathbf{x})\frac{1}{\sqrt{\mathbf{e}_k2\pi b_1}}e^{\mathbf{e}_1(\frac{a_1}{2b_1}x_1^2-\frac{1}{b_1}x_1u_1+\frac{d_1}{2b_1}u_1^2)}\mathbf{e}_1(\frac{a_1}{b_1}x_1-\frac{u_1}{b_1})dx_1
\\
&=\int_\mathbb{R}(\frac{u_1}{b_1}-\frac{a_1}{b_1}x_1)f(\mathbf{x})\mathbf{e}_1K^{\mathbf{e}_1}_{A_1}(x_1,u_1)dx_1.
\end{aligned}
\end{equation*}
Hence
\begin{equation}\label{PartialD}
\mathscr{L}^r_{\Lambda}(\frac{\partial f(\mathbf{x})}{\partial x_n})(\mathbf{u})=\int_{\mathbb{R}^n}(\frac{u_1}{b_1}-\frac{a_1}{b_1}x_1)f(\mathbf{x})\mathbf{e}_1\prod\limits_{k=1}^{n} K^{\mathbf{e}_k}_{A_k}(x_k,u_k)d^n\mathbf{x}.
\end{equation}

While
\begin{equation*}
\begin{aligned}
a_1\frac{\partial}{\partial u_1}\mathscr{L}^r_{\Lambda}(f)(\mathbf{u})&=a_1\frac{\partial}{\partial u_1}\int_{\mathbb{R}^n}f(\mathbf{x})\prod\limits_{k=1}^{n} K^{\mathbf{e}_k}_{A_k}(x_k,u_k)d^n\mathbf{x}
\\
&=\int_{\mathbb{R}^n}(\frac{a_1d_1}{b_1}u_1-\frac{a_1x_1}{b_1})f(\mathbf{x})\mathbf{e}_1\prod\limits_{k=1}^{n} K^{\mathbf{e}_k}_{A_k}(x_k,u_k)d^n\mathbf{x},
\end{aligned}
\end{equation*}
and
\begin{equation*}
\begin{aligned}
\mathbf{e}_1c_1u_1\mathscr{L}^r_{\Lambda}(f)(\mathbf{u})=\int_{\mathbb{R}^n}c_1u_1\mathbf{e}_1 f(\mathbf{x})\prod\limits_{k=1}^{n} K^{\mathbf{e}_k}_{A_k}(x_k,u_k)d^n\mathbf{x}.
\end{aligned}
\end{equation*}
According to equation (\ref{PartialD}), We can easily verified that $\mathscr{L}^r_{\Lambda}(\frac{\partial f(\mathbf{x})}{\partial x_1})(\mathbf{u})=(a_1\frac{\partial}{\partial u_1}-\mathbf{e}_1c_1u_1)\mathscr{L}^r_{\Lambda}(f)(\mathbf{u})$, where the property $a_1d_1-b_1c_1=1$ is used. The proof of equation (\ref{Partialn}) is similar to that of equation (\ref{Partial1}), and we will omit it here.
\end{proof}


\begin{thm}(Plancherel) Suppose $f(\mathbf{x})$ and $g(\mathbf{x}) \in L^1\bigcap L^2(\mathbb{R}^n,Cl(0,n))$, $F(\mathbf{u}):=\mathscr{L}^r_{\Lambda}(f)(\mathbf{u})$ and $G(\mathbf{u}):=\mathscr{L}^r_{\Lambda}(g)(\mathbf{u})$ are their CLCTs respectively, then we have
\begin{equation}\label{innerp1}
(f,g)_{L^2(\mathbb{R}^n,Cl(0,n))}=(F,G)_{L^2(\mathbb{R}^n,Cl(0,n))}.
\end{equation}
Furthermore, when $f=g$, they turn to
\begin{equation}\label{parseval1}
(f,f)_{L^2(\mathbb{R}^n,Cl(0,n))}=(F,F)_{L^2(\mathbb{R}^n,Cl(0,n))}.
\end{equation}
\end{thm}
\begin{proof} According to equation (\ref{innerp}),
\begin{equation}\label{Planchel1}
(F,G)_{L^2(\mathbb{R}^n,Cl(0,n))}=\int_{\mathbb{R}^n}F(\mathbf{u})\overline{G(\mathbf{u})}d^n\mathbf{u}.
\end{equation}
Inputting the integral formulaes of $F(\mathbf{u})$ and $G(\mathbf{u})$ into the above equation, one can derive
\begin{equation*}
\begin{aligned}
(F,G)&=\int_{\mathbb{R}^n}\bigg[\int_{\mathbb{R}^n}f(\mathbf{x})\prod\limits_{k=1}^{n} K^{\mathbf{e}_k}_{A_k}(x_k,u_k)d^n\mathbf{x}\bigg]\bigg[\overline{\int_{\mathbb{R}^n}g(\mathbf{y})\prod\limits_{k=1}^{n} K^{\mathbf{e}_k}_{A_k}(y_k,u_k)d^n\mathbf{y}}\bigg]d^n\mathbf{u}
\\
&=\int_{\mathbb{R}^n}\bigg[\int_{\mathbb{R}^{2n}}f(\mathbf{x})\prod\limits_{k=1}^{n} K^{\mathbf{e}_k}_{A_k}(x_k,u_k)\prod\limits_{k=0}^{n-1} K^{-\mathbf{e}_{n-k}}_{A_{n-k}}(y_{n-k},u_{n-k})\overline{g(\mathbf{y})}d^n\mathbf{x}d^n\mathbf{y}\bigg]d^n\mathbf{u}
\\
&=\int_{\mathbb{R}^{2n}}f(\mathbf{x})\bigg[\int_{\mathbb{R}^n}\prod\limits_{k=1}^{n} K^{\mathbf{e}_k}_{A_k}(x_k,u_k)\prod\limits_{k=0}^{n-1} K^{-\mathbf{e}_{n-k}}_{A_{n-k}}(y_{n-k},u_{n-k})d^n\mathbf{u}\bigg]\overline{g(\mathbf{y})}d^n\mathbf{x}d^n\mathbf{y}.
\end{aligned}
\end{equation*}
The interchanging of the integral in the above equation is due to the Fubini theorem which is insured by $f(\mathbf{x})$ and $g(\mathbf{x}) \in L^2(\mathbb{R}^n,Cl(0,n))$.
Furthermore, the integral in bracket can be turned into
\begin{equation*}
\begin{aligned}
\int_{\mathbb{R}^n}&\prod\limits_{k=1}^{n} K^{\mathbf{e}_k}_{A_k}(x_k,u_k)\prod\limits_{k=0}^{n-1} K^{-\mathbf{e}_{n-k}}_{A_{n-k}}(y_{n-k},u_{n-k})d^n\mathbf{u}=\int_{\mathbb{R}^{n-1}}\prod\limits_{k=1}^{n-1}K^{\mathbf{e}_k}_{A_k}(x_k,u_k)
\\
&\bigg[\int_{\mathbb{R}}K^{\mathbf{e}_n}_{A_n}(x_n,u_n)K^{-\mathbf{e}_n}_{A_n}(y_n,u_n)du_n\bigg]\prod\limits_{k=1}^{n-1} K^{-\mathbf{e}_{n-k}}_{A_{n-k}}(y_{n-k},u_{n-k})d^{n-1}\mathbf{u},
\end{aligned}
\end{equation*}
where
$\int_{\mathbb{R}}K^{\mathbf{e}_n}_{A_n}(x_n,u_n)K^{-\mathbf{e}_n}_{A_n}(y_n,u_n)du_n$
can be calculated as equation (\ref{Delta}) and equals $\delta (x_n-y_n)$.
Hence, the above equation is
\begin{equation*}
\int_{\mathbb{R}^n}\prod\limits_{k=1}^{n} K^{\mathbf{e}_k}_{A_k}(x_k,u_k)\prod\limits_{k=0}^{n-1} K^{-\mathbf{e}_{n-k}}_{A_{n-k}}(y_{n-k},u_{n-k})d^n\mathbf{u}=\prod\limits_{k=1}^n\delta (x_k-y_k),
\end{equation*}
and $(F,G)$ turns to
\begin{equation*}
\begin{aligned}
(F,G)&=\int_{\mathbb{R}^n}f(\mathbf{x})\bigg[\int_{\mathbb{R}^n}\prod\limits_{k=1}^n\delta (x_k-y_k)\overline{g(\mathbf{y})}d^n\mathbf{y}\bigg]d^n\mathbf{x}
\\
&=\int_{\mathbb{R}^n}f(\mathbf{x})\overline{g(\mathbf{x})}d^n\mathbf{x}=(f,g).
\end{aligned}
\end{equation*}
Let $f=g$, the Parseval theorem (\ref{parseval1}) is derived.
\end{proof}


\begin{thm}\label{inverset}(Inverse Theorem) Let $\Lambda$: $A_k=\left(
\begin{array}{cc}
a_k & b_k \\
c_k & d_k \\
\end{array}
\right)$
($k=1,2,\cdots,n$), suppose  $f(\mathbf{x})\in L^1(\mathbb{R}^n,Cl(0,n))$. Then the CLCT of $f$, $F(\mathbf{u}):=\mathscr{L}^r_{\Lambda}(f)(\mathbf{u})$, is an invertible transform and its inverse is
\begin{equation}\label{inversef}
f(\mathbf{x})=\mathscr{L}^{-1}_{\Lambda}(F)(\mathbf{x})=\int_{\mathbb{R}^n}F(\mathbf{u})\prod\limits_{k=0}^{n-1} K^{\mathbf{e}_{n-k}}_{{A^{-1}_{n-k}}}(u_{n-k}, x_{n-k})d^n\mathbf{u},
\end{equation}
where $A^{-1}_k=\left(
\begin{array}{cc}
d_k & -b_k \\
-c_k & a_k \\
\end{array}
\right),$ ($k=1,2,\cdots,n$).
\end{thm}
\begin{proof}
$f(\mathbf{x})\in L^1(\mathbb{R}^n,Cl(0,n))$, by straightforward computation, one can find
\begin{equation}\label{inversef1}
\begin{aligned}
\int_{\mathbb{R}^n}F(\mathbf{u})\prod\limits_{k=0}^{n-1} K^{\mathbf{e}_{n-k}}_{{A^{-1}_{n-k}}}(u_{n-k}, x_{n-k})d^n\mathbf{u}=\int_{\mathbb{R}^n}\bigg[\int_{\mathbb{R}^n}f(\mathbf{y})\prod\limits_{k=1}^{n} K^{\mathbf{e}_{k}}_{A_k}(y_k, u_k)d^n\mathbf{y}\bigg]
\\
\prod\limits_{k=0}^{n-1} K^{\mathbf{e}_{n-k}}_{{A^{-1}_{n-k}}}(u_{n-k},x_{n-k})d^n\mathbf{u}
\\
=\int_{\mathbb{R}^{2n}}f(\mathbf{y})\prod\limits_{k=1}^{n} K^{\mathbf{e}_{k}}_{A_k}(y_k, u_k)d^n\mathbf{y}\prod\limits_{k=0}^{n-1} K^{\mathbf{e}_{n-k}}_{{A^{-1}_{n-k}}}(u_{n-k},x_{n-k})d^n\mathbf{y}d^n\mathbf{u},
\\
=\int_{\mathbb{R}^n}f(\mathbf{y})\bigg[\int_{\mathbb{R}^n}\prod\limits_{k=1}^{n} K^{\mathbf{e}_{k}}_{A_k}(y_k, u_k)
\prod\limits_{k=0}^{n-1} K^{\mathbf{e}_{n-k}}_{{A^{-1}_{n-k}}}(u_{n-k},x_{n-k})d^n\mathbf{u}\bigg]d^n\mathbf{y}
\end{aligned}
\end{equation}
where $K^{\mathbf{e}_k}_{{A^{-1}_k}}(u_k,x_k)=\frac{1}{\sqrt{\mathbf{e}_k2\pi b_k}}e^{\mathbf{e}_k(-\frac{a_k}{2b_k}x_k^2+\frac{1}{b_k}x_ku_k-\frac{d_k}{2b_k}u_k^2)}$,
similar to equation (\ref{Delta}), the integral in bracket of above equation is equal to $\prod\limits_{k=1}^{n}\delta(x_k-y_k)$.
Thus equation (\ref{inversef1}) becomes
\begin{equation*}
\int_{\mathbb{R}^n}F(\mathbf{u})\prod\limits_{k=0}^{n-1} K^{\mathbf{e}_{n-k}}_{{A^{-1}_{n-k}}}(u_{n-k}, x_{n-k})d^n\mathbf{u}=\int_{\mathbb{R}^n}f(\mathbf{y})\prod\limits_{k=1}^{n}\delta(x_k-y_k)d^n\mathbf{y}=f(\mathbf{x}),
\end{equation*}
which completes the proof.
\end{proof}

\section{Analytic Signal in CLCT Domain and Envelop Detection}
\subsection{Analytic Signal in CLCT Domain}
Following the approach of generalizing the analytic signal to LCT domain, which was introduced by Fu and Li \cite{Fu08} and extended to quaternion algebra by Kou \cite{Kou16}, we supply the definition of generalized analytic signal which is associated with CLCT.

\begin{defn}Given a Clifford valued function $f(\mathbf{x})$ with $\mathbf{x}=(x_1,x_2,\cdots,x_n)$, its corresponding analytic signal $f_A(\mathbf{x})$ is defined as
\begin{equation}
f_A(\mathbf{x})=\int_{\mathbb{R}^n}F_A(\mathbf{u})\prod\limits_{k=0}^{n-1} K^{\mathbf{e}_{n-k}}_{{A^{-1}_{n-k}}}(u_{n-k}, x_{n-k})d^n\mathbf{u},
\end{equation}
where $F_A(\mathbf{u})$ is
\begin{equation}
F_A(\mathbf{u})=\prod\limits_{k=1}^{n}[1+sign(\frac{u_k}{b_k})]F(\mathbf{u}),
\end{equation}
and $F(\mathbf{u}):=\mathscr{L}^r_{\Lambda}(f)(\mathbf{u})$ is the CLCT of original function $f(\mathbf{x})$, $\Lambda:$ $A_k=\left(
\begin{array}{cc}
a_k & b_k \\
c_k & d_k \\
\end{array}
\right)\in \mathbb{R}^{2\times2}$
($k=1,2,\cdots,n$) be a set of parameter matrices  with $det(A_k)=1$ and $b_k\neq0$.
\end{defn}
When the dimension $n=1$, the analytic signal $f_A(t)$ is just what Fu introduced in \cite{Fu08}. While $n=2$, this analytic signal $f_A(x,y)$ is different from Kou's generalized quaternionic analytic signal, which is associated with two-sided QLCT.

Since we intend to investigate the envelop detection problems of 3-D images, we will focus on $n=3$ case. Due to the advantages of Clifford algebra $Cl(0,3)$, we will adopt this special Clifford algebra in the following part.

\begin{defn}
Let $\Lambda$: $A_k=\left(
\begin{array}{cc}
a_k & b_k \\
c_k & d_k \\
\end{array}
\right)\in \mathbb{R}^{2\times2}$
($k=1,2,3$) be a set of parameter matrices  with $det(A_k)=1$ and $b_k\neq0$, the right-sided CLCT of a signal $f\in L^1(\mathbb{R}^3,Cl(0,3))$ is a Clifford valued function $\mathscr{L}^r_{\Lambda}(f): \mathbb{R}^3\rightarrow Cl(0,3)$ defined as follows
\begin{equation}
\mathscr{L}^r_{\Lambda}(f)(\mathbf{u}):=\int_{\mathbb{R}^3}f(\mathbf{x}) K^{\mathbf{e}_1}_{A_1}(x_1,u_1)K^{\mathbf{e}_2}_{A_2}(x_2,u_2)K^{\mathbf{e}_3}_{A_3}(x_3,u_3)d^3\mathbf{x},
\end{equation}
where $K^{\mathbf{e}_1}_{A_1}(x_1,u_1)$, $K^{\mathbf{e}_2}_{A_2}(x_2,u_2)$ and $K^{\mathbf{e}_3}_{A_3}(x_3,u_3)$ are kernels of the CLCT and defined by
\begin{equation*}
K^{\mathbf{e}_k}_{A_k}(x_k,u_k):=\frac{1}{\sqrt{\mathbf{e}_k2\pi b_k}}e^{\mathbf{e}_k(\frac{a_k}{2b_k}x_k^2-\frac{1}{b_k}x_ku_k+\frac{d_k}{2b_k}u_k^2)}.
\end{equation*}
\end{defn}

\begin{defn}\label{analyticC}
Given a Clifford biquaternion valued function $f(\mathbf{x})$ with $\mathbf{x}=(x_1,x_2,x_3)$, its corresponding Clifford biquaternion analytic signal $f_A(\mathbf{x})$ is defined as
\begin{equation}
f_A(\mathbf{x})=\int_{\mathbb{R}^3}F_A(\mathbf{u})K^{\mathbf{e}_3}_{{A^{-1}_{3}}}(u_3, x_3)K^{\mathbf{e}_2}_{{A^{-1}_{2}}}(u_2, x_2)K^{\mathbf{e}_1}_{{A^{-1}_{1}}}(u_1, x_1)du_1du_2du_3,
\end{equation}
where $F_A(\mathbf{u})$ is
\begin{equation}
F_A(\mathbf{u})=[1+sign(\frac{u_1}{b_1})][1+sign(\frac{u_2}{b_2})][1+sign(\frac{u_3}{b_3})]F(\mathbf{u}),
\end{equation}
and $F(\mathbf{u}):=\mathscr{L}^r_{\Lambda}(f)(\mathbf{u})$ is the CLCT of original function $f(\mathbf{x})$, $\Lambda:$ $A_k=\left(
\begin{array}{cc}
a_k & b_k \\
c_k & d_k \\
\end{array}
\right)\in \mathbb{R}^{2\times2}$
($k=1,2,3$) be a set of parameter matrices with $det(A_k)=1$ and $b_k\neq0$.
\end{defn}

The analytic signal $f_A(\mathbf{x})$ defined by Definition \ref{analyticC} is a  Clifford biquaternion valued function, which can be written as
\begin{equation*}
f_A(\mathbf{x})=(p_0+\mathbf{i}p_1+\mathbf{j}p_2+\mathbf{k}p_3)+\epsilon(q_0+\mathbf{i}q_1+\mathbf{j}q_2+\mathbf{k}q_3).
\end{equation*}
From this function, one can derive three quaternions:
\begin{equation*}
f_{A_{\mathbf{e}_1\mathbf{e}_2}}=p_0+q_1\mathbf{e}_1+q_2\mathbf{e}_2+p_3\mathbf{e}_1\mathbf{e}_2,
\end{equation*}
\begin{equation*}
f_{A_{\mathbf{e}_2\mathbf{e}_3}}=p_0+q_2\mathbf{e}_2+q_3\mathbf{e}_3+p_1\mathbf{e}_2\mathbf{e}_3,
\end{equation*}
and
\begin{equation*}
f_{A_{\mathbf{e}_3\mathbf{e}_1}}=p_0+q_3\mathbf{e}_3+q_1\mathbf{e}_1+p_2\mathbf{e}_3\mathbf{e}_1.
\end{equation*}

For each of this quaternionic valued signals, one can write them into polar form representations, and obtain their modules which named the partial modules of the original Clifford biquaternion valued analytic signal $f_A(\mathbf{x})$ as following
\begin{equation*}
mod_{\mathbf{e}_1\mathbf{e}_2}=\sqrt{f_{A_{\mathbf{e}_1\mathbf{e}_2}}\overline{(f_{A_{\mathbf{e}_1\mathbf{e}_2}})}},
\end{equation*}
\begin{equation*}
mod_{\mathbf{e}_2\mathbf{e}_3}=\sqrt{f_{A_{\mathbf{e}_2\mathbf{e}_3}}\overline{(f_{A_{\mathbf{e}_2\mathbf{e}_3}})}},
\end{equation*}
\begin{equation*}
mod_{\mathbf{e}_3\mathbf{e}_1}=\sqrt{f_{A_{\mathbf{e}_3\mathbf{e}_1}}\overline{(f_{A_{\mathbf{e}_3\mathbf{e}_1}})}}.
\end{equation*}

\subsection{Envelop Detection of 3-D Images}
In \cite{Wang12}, Wang constructed a experiment platform for acquiring radio frequency (RF) ultrasound volume with a biopsy needle in it. The volume is a $128\times1280\times33$ pixels in lateral ($x_1$ axes), elevation ($x_2$ axes) and axial ($x_3$ axes) directions respectively. By
using classical 1-D analytic signal and novel 3-D analytic signal approaches, different resolution appears in ultrasound image envelop detection. Since 3-D analytic signal takes into account the information of the neighbouring scan lines, this novel approaches supplies better result than classical 1-D analytic signal approach. In this paper, we will take the same experiment platform and show the shortcoming of the analytic signal associated with CFT. We also propose the analytic signal associated with CLCT in Clifford biquaternion domain to this problem and show that by modifying the matrix parameters, one can obtain a satisfactory result.

\begin{figure}
\begin{minipage}[t]{0.33\linewidth}
\centering
\includegraphics[width=2in]{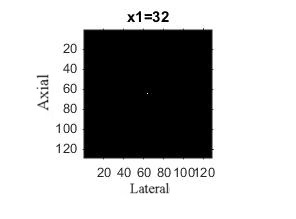}
\end{minipage}%
\begin{minipage}[t]{0.33\linewidth}
\centering
\includegraphics[width=2in]{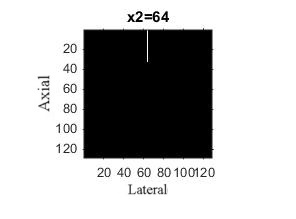}
\end{minipage}
\begin{minipage}[t]{0.33\linewidth}
\centering
\includegraphics[width=2in]{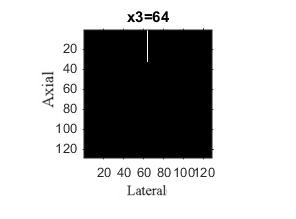}
\end{minipage}
\caption{original envelop of 3-D image.}
\end{figure}

Here we modify the volume to a $128\times128\times128$ pixels, and in the plane of $x_3=64$ there is a biopsy needle located at $x_2=64$ with $0\leq x_1\leq32$, which  can be easily found in Figure 1. The flow-process diagram of using the amplitude method of analytic signals is supplied in Figure 2.

\begin{figure}
\begin{tikzpicture}[node distance=1.0cm]
 \centering
\node[startstop](start){Inputting original data $f(x_1,x_2,x_3)$};
\node[process, below of = start, yshift = -0.5cm](pro1){Obtain analytic signal $f_A(x_1,x_2,x_3)$};
\node[process, below of = pro1, yshift = -0.5cm](pro2){Derive the polar form};
\node[process, below of = pro2, yshift = -0.5cm](pro3){Obtain the modulus $|f_A(x_1,x_2,x_3)|$};
\node[process, below of = pro3, yshift = -0.5cm](stop){draw the picture of $|f_A(x_1,x_2,x_3)|$};
\coordinate (point1) at (-3cm, -6cm);
\draw [arrow] (start)--(pro1);
\draw [arrow] (pro1)--(pro2);
\draw [arrow] (pro2)--(pro3);
\draw [arrow] (pro3)--(stop);
\end{tikzpicture}
\caption{ The flow-process diagram of using analytic signal's amplitude to envelop detection.}
 \end{figure}
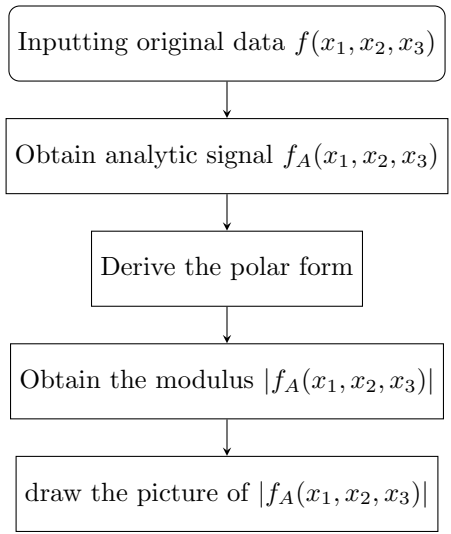

 Figure 3 shows the result by using Wang's analytic signal to detect the envelop of biopsy needle which locates at plane of $x_3=64$. Thank to take into account the information of the neighbouring scan lines, one can find the biopsy needle in the nearest planes besides in plane $x_3=64$, which are planes $x_3=63$ and $x_3=65$. Since it will induce many noise and lose information in converting RF ultrasound signal to the Brightness mode (B-mode) signal, this property will insure one can find a clear envelop in B-mode images.  But unfortunately, the envelops in these three planes distort badly, which are much longer than the true needle. This shortcoming can be well overcame by our approach based on the analytic signal associated with CLCT, which can be found in Figure 4. The envelop also can be found in planes $x_3=63$,$x_3=64$ and $x_3=65$, and the length is very close to that of true needle. Here the parameter matrices are $A_1=\left(
\begin{array}{cc}
1 & 10 \\
-0.5 & 0.5 \\
\end{array}
\right)$, $A_2=\left(
\begin{array}{cc}
1 & 10 \\
-0.5 & 0.5 \\
\end{array}
\right)$ and $A_3=\left(
\begin{array}{cc}
1000 & 1 \\
-0.1 & 0.0009 \\
\end{array}
\right)$, and they are found by two steps: first, to the 2-D image, plane $x_3=64$, we try several times and get good result. Furthermore, we slightly modify the parameters of $A_3$ and find a better result. Since there are two many parameters in these three matrices, we can not insure these three matrices are the best parameters. Due to the fact, Wang's method is a special case of our approach, we can insure that our approach will supply a better result by choosing proper parameter matrices.

Furthermore, we also use the envelop method of monogenic signal, which was introduced by Yang \cite{Yang17} and is another higher dimensional generalization of analytic signal, to detect the envelop of this image. Figure 4 shows that it supplies an accurate envelop in every planes but one can find this envelop in every $x_3$ plane, which means that the image will be contaminated by the background signal badly.

\begin{figure}
\begin{minipage}[t]{0.33\linewidth}
\centering
\includegraphics[width=2in]{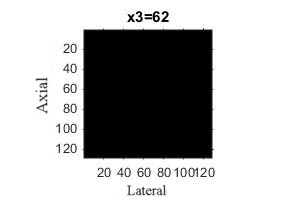}
\end{minipage}%
\begin{minipage}[t]{0.33\linewidth}
\centering
\includegraphics[width=2in]{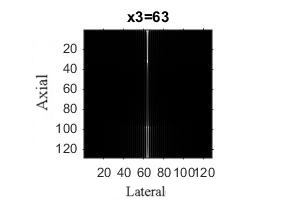}
\end{minipage}
\begin{minipage}[t]{0.33\linewidth}
\centering
\includegraphics[width=2in]{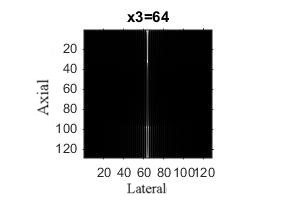}
\end{minipage}
\begin{minipage}[t]{0.33\linewidth}
\centering
\includegraphics[width=2in]{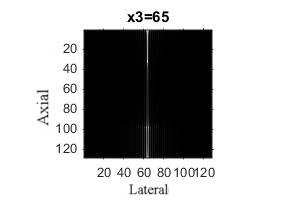}
\end{minipage}
\begin{minipage}[t]{0.33\linewidth}
\centering
\includegraphics[width=2in]{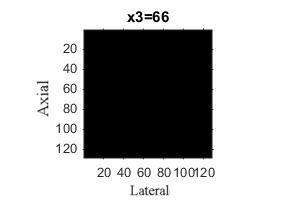}
\end{minipage}
\caption{Wang's analytic signal associated with CFT \cite{Wang12}.}
\end{figure}

\begin{figure}
\begin{minipage}[t]{0.33\linewidth}
\centering
\includegraphics[width=2in]{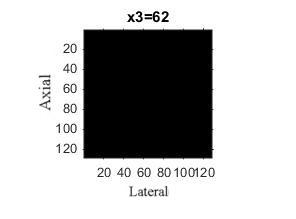}
\end{minipage}%
\begin{minipage}[t]{0.33\linewidth}
\centering
\includegraphics[width=2in]{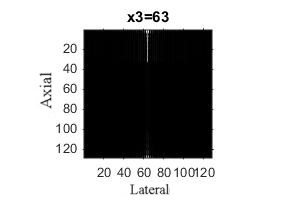}
\end{minipage}
\begin{minipage}[t]{0.33\linewidth}
\centering
\includegraphics[width=2in]{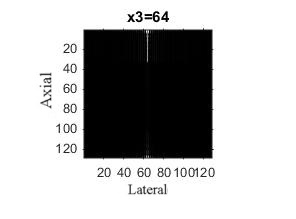}
\end{minipage}
\begin{minipage}[t]{0.33\linewidth}
\centering
\includegraphics[width=2in]{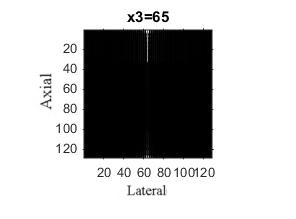}
\end{minipage}
\begin{minipage}[t]{0.33\linewidth}
\centering
\includegraphics[width=2in]{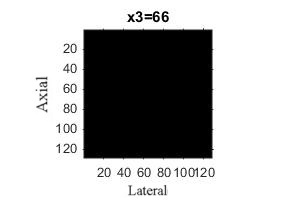}
\end{minipage}
\caption{This paper's method: analytic signal associated with CLCT, where\qquad \qquad $A_1=(1,10;-0.05,0.5)$,
$A_2=(1,10;-0.05,0.5)$ and $A_3=(1000,1;-0.1,0.0009)$.}
\end{figure}

\begin{figure}
\begin{minipage}[t]{0.33\linewidth}
\centering
\includegraphics[width=2in]{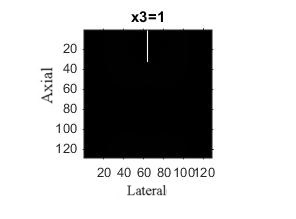}
\end{minipage}%
\begin{minipage}[t]{0.33\linewidth}
\centering
\includegraphics[width=2in]{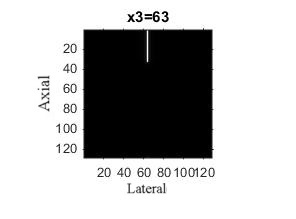}
\end{minipage}
\begin{minipage}[t]{0.33\linewidth}
\centering
\includegraphics[width=2in]{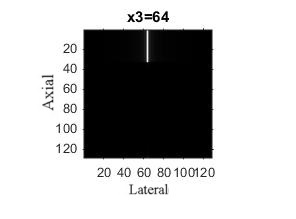}
\end{minipage}
\begin{minipage}[t]{0.33\linewidth}
\centering
\includegraphics[width=2in]{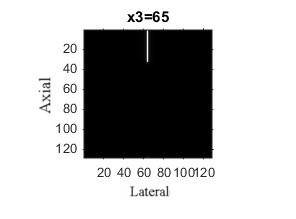}
\end{minipage}
\begin{minipage}[t]{0.33\linewidth}
\centering
\includegraphics[width=2in]{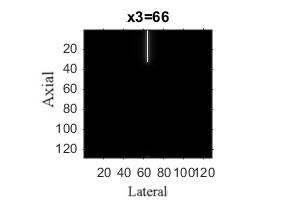}
\end{minipage}
\begin{minipage}[t]{0.3\linewidth}
\centering
\includegraphics[width=2in]{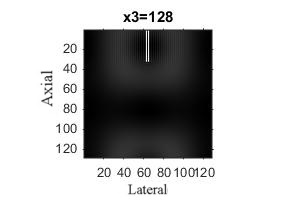}
\end{minipage}
\caption{Monogenic signal approach of \cite{Yang17}.}
\end{figure}

\section{Conclusion}

Inspired by the success of 3-D biquaternionic analytic signal in detecting 3-D ultrasound images envelop, we generalize this approach to LCT domain. Thanks to the special structure of the Clifford biquaternion algebra, this generalized analytic signal has three partial modules where $mod_{\mathbf{e}_1\mathbf{e}_2}(f_A)$ corresponds the envelop of 3-D images in $x_3$ planes. Real valued synthetic 3-D image is introduced to test the valid of this novel envelop detection method. Comparing with the classical 1-D analytic signal method and Wang's generalized biquaternionic analytic signal method even the envelop method of monogenic signal, our approach supplies a best result. Furthermore, our method can derive a better resolution by modifying the  parameter matrices and adding post processing such as average filter processing.


\subsection*{Acknowledgment}
The first author acknowledges financial support from the PhD research startup foundation of Hubei University of Technology No. BSQD2019052. Partial support by the Foundation for Science and Technology from Department of Education of Hubei province (B2019047).  Partial support by the Macao Science and Technology Development Fund 0085/2018/A2.

\end{document}